\documentclass[11pt]{amsart}
\usepackage{amsmath}
\usepackage{amsthm}
\usepackage{latexsym}
\usepackage{amssymb}
\usepackage{xspace}
\usepackage{amscd}
\usepackage{enumerate}
\usepackage{lscape}
\usepackage{color}
\usepackage[normalem]{ulem}
\usepackage{amsmath}
\usepackage{amsmath, pb-diagram}
\usepackage[all]{xy}
\usepackage{amssymb}
\usepackage{amscd}
\newtheorem{thm}{Theorem}[section]
 \newtheorem{cor}[thm]{Corollary}
 \newtheorem{lem}[thm]{Lemma}
 \newtheorem{prop}[thm]{Proposition}
 \theoremstyle{definition}

 \newtheorem{rem}[thm]{Remark}
  
 \numberwithin{equation}{section}

\DeclareMathOperator{\End}{End}
\DeclareMathOperator{\Hom}{Hom}
\DeclareMathOperator{\Ker}{Ker}
\DeclareMathOperator{\im}{Im}

\begin{document}

\author{ Abyzov Adel Nailevich, Truong Cong Quynh
\\
and Tran Hoai Ngoc Nhan}
\address{Department of Algebra and Mathematical Logic, Kazan (Volga Region) Federal University, 18 Kremlyovskaya str., Kazan, 420008 Russia}
\email{aabyzov@ksu.ru, Adel.Abyzov@ksu.ru}
\address{Department of Mathematics, Danang University, 459 Ton Duc Thang, Danang city, Vietnam}
 
\email{tcquynh@dce.udn.vn; tcquynh@live.com}
\address{Department of Algebra and Mathematical Logic, Kazan (Volga Region) Federal University, 18 Kremlyovskaya str., Kazan, 420008 Russia}
\email{tranhoaingocnhan@gmail.com}
\title[C3 and D3 modules]{On  classes of  C3 and D3 modules}

\begin{abstract} The aim of this paper is to  study the notions of
 $\mathcal{A}$-C3 and  $\mathcal{A}$-D3 modules for some class $\mathcal{A}$ of right modules.  Several characterizations of these modules
are provided and used to describe some well-known classes of rings and modules. For
example,    a regular right $R$-module $F$  is a $V$-module if and only if every $F$-cyclic module $M$  is an $\mathcal{A}$-C3 module where $\mathcal{A}$ is the class of all simple submodules of $M$. Moreover,  let $R$ be a right artinian ring  and $\mathcal{A}$,   a class of  right $R$-modules with local endomorphisms, containing  all simple right $R$-modules and closed under isomorphisms.  If all right $R$-modules are $\mathcal{A}$-injective, then    $R$ is a serial artinian ring with $J^{2}(R)=0$ if and only if every $\mathcal{A}$-C3 right $R$-module is quasi-injective,  if and only if every  $\mathcal{A}$-C3  right $R$-module is C3.

\end{abstract}

\maketitle

\bigskip

\bigskip

\section{Introduction and notation.}

The study of modules with summand intersection property was motivated by the following result of Kaplansky: every free module over a commutative principal ideal ring has the summand intersection property (see \cite[Exercise 51(b)]{KP}). A module $M$ is said to have the \emph{summand intersection property}  if the intersection of any two direct summands of $M$ is a direct summand of $M$. This definition is introduced  by Wilson \cite{Wi}.
Dually, Garcia \cite{G89} consider the summand sum property. A module $M$ is said to have the \emph{summand sum property}   if the sum of any two direct summands is a direct summand of $M$. These properties have been  studied by several  authors (see  \cite{AT14, AH02, HHO05,H89, QKT},...).  Moreover, the  classes of C3-modules and D3-modules have recently studied by Yousif et al. in \cite{Amin,YAI}. Some characterizations of semisimple rings and regular rings and other  classes of rings are studied via C3-modules and D3-modules. On the other hand, several authors investigated  some properties of  generalizations of   C3-modules and D3-modules in \cite{CIYZ, IKQY}; namely, simple-direct-injective modules  and simple-direct-projective modules.  A right $R$-module $M$ is called a $C3$-module if, whenever $A$ and $B$ are submodules of $M$ with $A \subset_d M$, $B \subset_d M$ and $A\cap B=0$, then $A\oplus B\subset_d  M $.   $M$ is called \textit{simple-direct-injective} in \cite{CIYZ} if the submodules $A$ and $B$ in the
above definition are simple. Dually, $M$ is called a $D3$-module if, whenever $M_{1}$ and $M_{2}$ are
direct summands of $M$ and \ $M=M_{1}+M_{2},$ then $M_{1}\cap M_{2}$ is a
direct summand of $M$.  $M$ is called \textit{simple-direct-projective} in \cite{IKQY} if the submodules $M_1$ and $M_2$ in the above definition are maximal.

In Section 2, we introduce the notions of    $\mathcal{A}$-C3 modules and  $\mathcal{A}$-D3 modules, where $\mathcal{A}$ is a class of right modules over the ring $R$ and  closed under isomorphisms.   It is shown that if each factor module of $M$ is $\mathcal {A}$-injective, then $M$ is an   $\mathcal{A}$-D3 module if and only if $M$ satisfies D2 for the class $\mathcal {A}$,  if and only if $M$ have the summand intersection property  for the class $\mathcal {A}$  in Proposition \ref{4.2}. On the other hand, if every submodule of $M$ is  $\mathcal{A}$-projective, then $M$ is an $\mathcal{A}$-C3 module if and only if $M$ satisfies C2 for the class $\mathcal {A}$,  if and only if $M$  have the  summand sum property for the class $\mathcal {A}$  in Proposition \ref{4.7}. Some well-known properties of other modules  are obtained from these results.

In Section 3, we provide some characterizations of serial artinian rings and semisimple artinian rings.  The Theorem \ref{thm:4.2q} and Theorem \ref{thm:4.4q} are indicated that let $R$ be a right artinian ring and  $\mathcal{A}$,   a class of  right $R$-modules with local endomorphisms, containing  all simple right $R$-modules and closed under isomorphisms:

\begin{enumerate}
 \item If all right $R$-modules are $\mathcal{A}$-injective, the following conditions are equivalent for a ring $R$:
\begin{enumerate}
\item [(i)] $R$ is a serial artinian ring with $J^{2}(R)=0$.
\item [(ii)] Every $\mathcal{A}$-C3 right $R$-module is quasi-injective.
\item [(iii)] Every $\mathcal{A}$-C3  right $R$-module is $C3$.
\end{enumerate}
\item  If all right $R$-modules are $\mathcal{A}$-projective, then the following conditions are equivalent for a ring $R$:

\begin{enumerate}
\item [(i)] $R$ is a serial artinian ring with $J^{2}(R)=0$.
\item [(ii)] Every $\mathcal{A}$-D3 right $R$-module is quasi-projective.
\item [(iii)] Every $\mathcal{A}$-D3 right $R$-module is $D3$.
\end{enumerate}
\end{enumerate}
Moreover, we give an equivalent condition for a regular $V$-module. It is shown that a regular right $R$-module $F$  is a $V$-module if and only if every $F$-cyclic module is simple-direct-injective in Theorem \ref{thm:4.10a}. It is an extension the result of rings to modules.

Throughout this paper $R$ denotes an associative ring with identity, and modules will be unitary right $R$-modules. The Jacobson radical   ideal in $R$ is denoted by $J(R)$.   The notations $N \leq M$, $N \leq_e M$, $N \unlhd M$, or $N \subset_d M$ mean that $N$ is a submodule, an essential submodule, a fully invariant submodule, and  a direct summand of $M$,  respectively. Let $M$ and $N$ be right $R$-modules. $M$ is called $N$-injective  if for any right $R$-module $K$ and any monomorphism $f: K\to N$, the induced  homomorphism $\Hom(N,M)\to \Hom(K,M)$ by $f$  is an epimorphism. $M$ is called $N$-projective   if for any right $R$-module $K$ and any epimorphism $f: N\to K$, the induced homomorphism $\Hom(M,N)\to \Hom(M,K)$ by $f$  is an epimorphism. Let $\mathcal{A}$ be a class of right modules over the ring $R$. $M$ is called {\it $\mathcal{A}$-injective} ({\it $\mathcal{A}$-projective}) if $M$ is $N$-injective (resp.,  $N$-projective) for all $N\in \mathcal{A}.$ We refer to \cite{AF}, \cite{DHSW}, \cite{MM}, and \cite{W} for all the
undefined notions in this paper.

\section{On $\mathcal{A}$-C3  modules and $\mathcal{A}$-D3 modules}

Let $\mathcal{A}$ be a class of right modules over a ring $R$ and  closed under
isomorphisms. We call that a right $R$-module $M$ is  an $\mathcal{A}$-C3 module if, whenever $A\in \mathcal{A}$ and $B\in \mathcal{A}$ are submodules of $M$ with $A \subset_d M$, $B \subset_d M$ and $A\cap B=0$, then $A\oplus B\subset_d  M $.   Dually, $M$ is an $\mathcal{A}$-D3 module if,  whenever $M_{1}$ and $M_{2}$ are direct summands of $M$ with  $M/M_1, M/M_2 \in \mathcal{A}$ and \ $M=M_{1}+M_{2},$ then $M_{1}\cap M_{2}$ is a direct summand of $M$.

\begin{rem} Let $M$ be a right $R$-module and $\mathcal{A}$,  a class of right $R$-modules.
\begin{enumerate}
\item  If $M$ is a   C3 (D3) module, then $M$ is an  $\mathcal{A}$-C3 (resp., $\mathcal{A}$-D3) module.
\item  If $\mathcal{A}={\rm Mod}-R$, then   $\mathcal{A}$-C3 modules ($\mathcal{A}$-D3 modules) modules are precisely the   C3 modules (resp.,  D3) modules.
\item  If $\mathcal{A}$ is the  class of all simple submodules of $M$, then $\mathcal{A}$-C3 ($\mathcal{A}$-D3) modules  are precisely the simple-direct-injective (resp., simple-direct-projective) modules and studied in \cite{CIYZ, IKQY}.
\item  If $\mathcal{A}$ is a class of injective right $R$-modules, then $M$ is always an $\mathcal{A}$-C3  module.
\item  If $\mathcal{A}$ is a class of  projective  right $R$-modules, then $M$ is always an $\mathcal{A}$-D3  module.
\end{enumerate}
\end{rem}
\begin{lem}\label{lem1}
Let $\mathcal{A}$ be a class of right $R$-modules and  closed under
isomorphisms. Then every summand of an  $\mathcal{A}$-C3 module ($\mathcal{A}$-D3 module) is also an    $\mathcal{A}$-C3 module (resp.,  $\mathcal{A}$-D3 module).

\end{lem}
\begin{proof}
The proof is straightforward.
\end{proof}
\begin{prop}\label{lemQ}
Let $\mathcal{A}$ be a class of right $R$-modules and  closed under
direct summands. Then the following conditions are equivalent for a module $M$:
\begin{enumerate}
\item $M$ is an $\mathcal{A}$-C3 module.
\item If  $A\in \mathcal{A}$ and $B\in \mathcal{A}$ are submodules of $M$ with $A \subset_d M$, $B \subset_d M$ and $A\cap B=0$, there exist submodules $A_1$ and $B_1$ of   $M$ such that  $M=A\oplus B_1=A_1\oplus B$ with $A\leq A_1$ and $B\leq B_1$.
\item If  $A\in \mathcal{A}$ and $B\in \mathcal{A}$ are submodules of $M$ with $A \subset_d M$, $B \subset_d M$ and $A\cap B\subset_d M$, then  $A+B\subset_d M$.
\end{enumerate}
\end{prop}
\begin{proof}
It is similar to  the proof of Proposition 2.2 in \cite{Amin}.
\end{proof}
Dually Proposition \ref{lemQ}, we have the following proposition.
\begin{prop}
Let $\mathcal{A}$ be a class of right $R$-modules and  closed under
isomorphisms.  Then the following conditions are equivalent for a module $M$:
\begin{enumerate}
\item $M$ is an $\mathcal{A}$-D3 module.
\item If  $M/A, M/B\in \mathcal{A}$ with $A \subset_d M$, $B \subset_d M$ and $M=A+ B$, then  $M=A\oplus B_1=A_1\oplus B$ with $A_1\leq A$ and $B_1\leq B$.
\item If  $M/A, M/B\in \mathcal{A}$ with  $A \subset_d M$, $B \subset_d M$ and $A+ B\subset_d M$, then  $A\cap B\subset_d M$.
\end{enumerate}
\end{prop}

Let $f: A\to B$ be a homomorphism. We denote by $\langle f \rangle$ the submodule of $A\oplus B$ as follows: $$ \langle f \rangle=\{a+f(a)\ | \ a\in A\}.$$
\vskip 0.2cm
 The following result is  proved in Lemma 2.6 of \cite{KMO}.
\begin{lem}\label{lem:addd} Let $M=X\oplus Y$ and $f: A\to Y$, a homomorphism with $A\leq X$. Then the following conditions hold
\begin{enumerate}
\item $A\oplus Y=\langle f \rangle\oplus Y$.
\item $\Ker(f)=X\cap \langle f \rangle$.
\end{enumerate}
\end{lem}
\begin{prop}\label{pro:add1} Let $M$ be an $\mathcal{A}$-D3 module with $\mathcal{A}$  a class of right $R$-modules and  closed under isomorphisms and summands.  If  $M=M_1\oplus M_2$ and  $f:M_1\to M_2$ is a homomorphism  with $\im(f)\subset_d M_2$ and $\im(f)\in \mathcal{A}$, then $\Ker(f)$ is a direct summand of $M_1$.
\end{prop}
\begin{proof} Assume that $M=M_1\oplus M_2$  and a homomorphism $f: M_1\to M_2$ with $\im(f)\subset_d M_2$ and $\im(f)\in \mathcal{A}$. Call $M':=M_1\oplus \im(f)$. Then $M'$ is a direct summand of $M$ and so is  an $\mathcal{A}$-D3 module.  It follows that $M'=M_1\oplus \im(f)=\langle f \rangle\oplus \im(f)$ by Lemma \ref{lem:addd}. It is easily to check  $M' / M_1, M' / \langle f \rangle \in \mathcal {A}$ and $M'=M_1+\langle f \rangle$. As $M'$ is  an $\mathcal{A}$-D3 module  and by  Lemma \ref{lem:addd}, $\langle f \rangle \cap M_1=\Ker(f)$ is a direct summand of $M'$. Thus $\Ker(f)$ is a direct summand of $M_1$.
\end{proof}

\begin{prop}\label{4.2} Let $M$ be   a right $R$-module and $\mathcal{A}$,  a class of right $R$-modules and  closed under isomorphisms and summands.
If each factor module of $M$ is  $\mathcal {A}$-injective, then the following conditions are equivalent:
\begin{enumerate}
\item For any two direct summands $M_1, M_2$ of $M$ such that $M/M_1, M/M_2 \in \mathcal {A}$,    $M_1\cap M_2$ is a direct summand of $M$.
\item $M$ is an $\mathcal{A}$-D3 module.
\item Any submodule $N$ of $M$ such that the factor module $M/N\in \mathcal {A}$ is isomorphic to a direct summand of $M$, is a direct summand of $M$.
\item For any decomposition $M=M_1\oplus M_2$ with $M_2\in \mathcal{A}$, then every homomorphism $f:M_1\to M_2$ has the kernel a direct summand of $M_1$.
\item  Whenever  $X_1,\ldots, X_n$ are direct summands of $M$ and $M/X_1,\dots, M/X_n \in \mathcal {A},$ then $\cap_{i=1}^nX_i$ is a direct summand of $M$.
\end{enumerate}

\end{prop}
\begin{proof}
$(2) \Rightarrow (1)$. Let $M_1, M_2$ be direct summands of $M$ such that $M / M_1, M / M_2 \in \mathcal {A}$. Then $M = M_1 \oplus M_1'$. Without loss of generality we can assume that $M_2 \nsubseteq M_1, M_2 \nsubseteq M_1'$. From our  assumption, $\pi (M_2)$ is a direct summand of $M_1 '$.  Then we can write  $M_1 '= \pi (M_2) \oplus M_1''$ for some $M_1''\leq M_1'$. Since the class $\mathcal{A}$ is closed under direct summands, $M_1''\in \mathcal{A}$. It is easy to see that $M_1 + M_1''$ is a direct summand of $M$. We have  $M / (M_1 + M_1'') \in \mathcal {A}$ and $M_1 + M_1'' + M_2 = M$. It follows that  $M_1 \cap M_2 = (M_1 + M_1'') \cap M_2$ is a direct summand of $M$.

$(3)\Rightarrow(2)$. It is obvious.

$(1)\Rightarrow(4)$. Assume that $M=M_1\oplus M_2$ with $M_2\in \mathcal{A}$ and a homomorphism $f: M_1\to M_2$. It follows that $M=M_1\oplus M_2=\langle f \rangle\oplus M_2$ by Lemma \ref{lem:addd}. Note that $M / M_1, M / \langle f \rangle \in \mathcal {A}$. By (1) and Lemma \ref{lem:addd}, $\langle f \rangle \cap M_1=\Ker(f)$ is a direct summand of $M$. Thus $\Ker(f)$ is a direct summand of $M_1$.

$(4)\Rightarrow(3)$. Let $M_1, M_2$ be submodules of $M$  such that $M = M_1 \oplus A$, $M / M_2 \cong A$ and $A \in \mathcal {A}$. Call $\pi_1: M\to M_1$ and $\pi_2: M\to A$ the projections.  By the hypothesis, $\pi_2 (M_2)$ is a direct summand of $A$   and hence   $A = \pi_2 (M_2) \oplus B$ for some submodule $B$ of $A$. Call $p: M\to M/M_2$ the canonical projection and isomorphism  $\phi: M/M_2\to A$. Take the homomorphism $f=\phi\circ(p|_{M_1}): M_1\to A$. It follows that $\Ker (f) = M_1 \cap M_2$.   By (4),  $\Ker (f) = M_1 \cap M_2$ is a direct summand of $M_1$. Call   $N_1$ a submodule of  $M_1$ with  $M_1 = N_1 \oplus (M_1 \cap M_2)$. Note that $M_1+M_2=M_1\oplus \pi_2(M_2)$ and $N_1\cap M_2=0$. This gives that
$$\begin{array}{llll}M &= M_1\oplus \pi_2(M_2)\oplus  B\\
&=(M_1+M_2)\oplus  B\\
&=[N_1 \oplus (M_1 \cap M_2)+M_2] \oplus  B=(N_1 +M_2) \oplus  B\\
&=(N_1\oplus M_2)\oplus B.
\end{array}$$

$(1) \Rightarrow (5)$.
We prove this by induction on $n$.
When $ n = 2 $,  the assertion is true from (1). Suppose that the assertion is true for $n = k.$
Let $X_1, X_2, \ldots, X_ {k + 1}$ be  summands of $ M $ and $M / X_1, M/X_2, \ldots, M / X_{k + 1} \in \mathcal{A}. $  We can write  $M=\cap_{i=1}^kX_i\oplus N$ for some submodule $N$ of  $M$. Without loss of generality we can assume that $\cap_{i=1}^kX_i\nsubseteq X_{k+1}$. Let $f:M\rightarrow M/X_{k+1}$ be the natural projection. Then $(\cap_{i=1}^kX_i)/[(\cap_{i=1}^kX_i)\cap X_{k+1}]$ is $\mathcal{A}$-injective, and therefore, it is isomorphic to a direct summand of $M/X_{k+1}\in  \mathcal{A}$. This gives that  $\cap_{i=1}^kX_i/\cap_{i=1}^{k+1}X_i$ is isomorphic to a direct summand of $M$ and
$$
M/(\cap_{i=1}^{k+1}X_i\oplus N)=(\cap_{i=1}^kX_i\oplus N)/(\cap_{i=1}^{k+1}X_i\oplus N)\in \mathcal{A}.
$$

Since the equivalence of (1) and (3), $(\bigcap\limits_{i=1}^{k+1}X_i)\oplus N$ is a direct summand of $M$. Thus $\bigcap\limits_{i=1}^{k+1}X_i$ is a direct summand of $M$.
\end{proof}

\begin{cor} The following conditions are equivalent for a module $M$:
\begin{enumerate}
\item If $M/A$ is  a semisimple module and  $B$, a submodule of $M$ with $M/A\cong B\subset_d M$, then $A \subset_d M$.
\item For any two direct summands $A, B$ of $M$ with  $M/A$ and $M/B$ are semisimple modules, then  $A\cap B\subset_d M$.
\item  For any two direct summands $A, B$ of $M$ such that $M/A, M/B$ are  semisimple modules and $A+B= M$,  then  $A\cap B$ is a direct summand of $M$.
\item Whenever  $X_1, X_2,\ldots, X_n$ are direct summands of $M$ and $M/X_1, M/X_2,\ldots, M/X_n$ are  semisimple modules, then $\cap_{i=1}^nX_i$ is a direct summand of $M$.
\end{enumerate}
\end{cor}
\begin{cor} Let $P$ be a quasi-projective module. If $X_1,\ldots, X_n$ are summands of $P$ and $P/X_1,\ldots, P/X_n$ are semisimple modules, then $\cap_{i=1}^nX_i$ is a direct summand of $P$.
\end{cor}

\begin{cor} The following conditions are equivalent for a module $M$:
\begin{enumerate}
\item For any maximal submodule $A$ of $M$  and any submodule $B$ of $M$ such that  $M/A\cong B\subset_d M,\ A \subset_d M$.
\item For any two maximal summands $A, B$ of $M,$  $A\cap B\subset_d M$.
\item If $M/A$ is a finitely generated semisimple module with $M/A\cong B\subset_d M$, then $A\subset_d M$.
\item Whenever  $X_1, X_2,\ldots, X_n$ are maximal summands of $M$,  then $\cap_{i=1}^nX_i$ is a direct summand of $M$.
\end{enumerate}
\end{cor}
\begin{proof}
$(1)\Leftrightarrow (2) \Leftrightarrow (4)$. Follow from Proposition \ref{4.2}.

$(3)\Rightarrow (1)$. Clearly.

$(1)\Rightarrow (3)$. Assume that $M/A$ is a finitely generated semisimple module and isomorphic to a direct summand of $M$. Write  $M/A = M_1/A \oplus \cdots \oplus M_n/A$ with simple submodules $M_i/A$ of $M/A$. Then
$M_i \cap (\sum_{j \ne i} M_j) = A$ for all $i = 1,2 \ldots, n$.  For any subset $\{i_1,i_2,\dots,i_{n-1}\}$ of the set $I:=\{1,2,\dots,n\}$, it is easily to see that $$M/(M_{i_1}+M_{i_2}+\cdots+M_{i_{n-1}})\simeq M_k/A$$ for some $k\in I\setminus \{i_1,i_2,\dots,i_{n-1}\}$. It follows that $M/(M_{i_1}+M_{i_2}+\cdots+M_{i_{n-1}})$ is isomorphic to a simple summand of $M$. By (1), $M_{i_1}+M_{i_2}+\cdots+M_{i_{n-1}}$ is a maximal summand of $M$. On the other hand,  we can check that $$A=\bigcap\limits_{\{i_1,i_2,\dots,i_{n-1}\}\subset I}(M_{i_1}+M_{i_2}+\cdots+M_{i_{n-1}}).$$

So, by (4), $A$ is a direct summand  of $M$.
\end{proof}

\begin{prop}\label{pro:ad2} Let $M$ be  an $\mathcal{A}$-C3 module with  $\mathcal{A}$  a class of right $R$-modules and  closed under isomorphisms and summands.  If $M=A_1\oplus A_2$  and $f:A_1\to A_2$ is a homomorphism with $\Ker(f)\in \mathcal{A}$ and $\Ker(f)\subset_d A_1$, then $\im(f)$ a direct summand of $A_2$.
\end{prop}
\begin{proof}Let  $f: A_1 \rightarrow  A_2$ be an $R$-homomorphism with $\Ker(f)\in \mathcal{A}$. By the hypothesis,  there exists a decomposition   $A_1 = \Ker(f)\oplus  B$ for a  submodule $B $ of $A_1$. Then $B\oplus  A_2$ is a direct summand of $M$. Note that every direct summand of an $\mathcal{A}$-C3 module is also an $\mathcal{A}$-C3 module. Hence   $B\oplus A_2$ is an $\mathcal{A}$-C3 module. Let  $g=f|_B: B\to A_2$. Then $g$ is a monomorphism and  $\im(g)=\im(f)$.   It is easy to see that   $B\oplus A_2=\langle g \rangle \oplus A_2$,    $\langle g \rangle\cap B=0$ and $\langle g \rangle\simeq  B$. Note that $B, \langle g \rangle\in \mathcal{A} $.    As  $B\oplus A_2$ is an $\mathcal{A}$-C3 module,   $B\oplus \langle g \rangle$ is a direct summand of $B\oplus A_2$.  Thus  $B\oplus \langle g \rangle=B\oplus \im(g)$, which implies that  $\im(g)$ or $\im(f)$ is a direct summand of $A_2$.
\end{proof}

\begin{prop}\label{4.5} Let $M$ be  a right $R$-module and $\mathcal{A}$,   a class of right $R$-modules and  closed under isomorphisms and summands.
If every submodule of $M$ is $\mathcal{A}$-projective,  the following conditions are equivalent:
\begin{enumerate}
\item For any two  direct summands $M_1, M_2$ of $M$ such that $M_1, M_2 \in \mathcal{A}$,   $M_1 + M_2$ is a direct summand  of $M$.
\item $M$ is an $\mathcal{A}$-C3 module.
\item For any decomposition $M=A_1\oplus A_2$ with $A_1\in \mathcal{A}$, then every homomorphism $f:A_1\to A_2$ has the image  a direct summand of $A_2$.
\end{enumerate}
\end{prop}
\begin{proof} $(1) \Rightarrow (2)$  is obvious.

$(2) \Rightarrow (3)$  Let  $f: A_1 \rightarrow  A_2$ be an $R$-homomorphism with $A_1\in \mathcal{A}$. By the hypothesis,  $\Ker(f)$ is a direct summand of   $A_1$. The rest of proof is followed from Proposition \ref{pro:ad2}.

$(3) \Rightarrow (1)$  Let $N$ and $K$ be  direct  summands of $M$ such that $N, K \in \mathcal{A}$. Write $M = N\oplus N'$ and $M = K\oplus K'$ for some submodules  $N', K'$ of $M$.    Consider  the canonical projections $\pi_{K}: M\to K$ and  $\pi_{N'}: M\to N'$. Let  $A=\pi_{N'}(\pi_K(N))$.    Then $ A =(N + K)\cap (N +K') \cap N'$  is a direct summand of $M$ by (3).  Write  $M = A\oplus L$ for some  submodule  $L$ of  $M$. Clearly, $$(N + K) \cap  [(N + K') \cap  (N'\cap L)]=0.$$ Hence,   $N'=A\oplus (N'\cap L)$ and $M=(N\oplus A)\oplus (N'\cap L)$. Since $A\leq N+K$ and $A\leq N+K'$, we get  $$N+K=(N\oplus A)\cap [(N+K)\cap (N'\cap L)]$$ and $$N+K'=(N\oplus A)\cap [(N+K')\cap (N'\cap L)].$$
They imply
$$\begin{array}{llll}
M&=N+K'+K\\
&=(N\oplus A)+[(N+K)\cap (N'\cap L)]+[(N+K')\cap (N'\cap L)]\\
&\leq (N+K)+[(N+K')\cap (N'\cap L)].
\end{array}
$$
Thus $M=(N+K)\oplus [(N+K')\cap (N'\cap L)$.
\end{proof}

\begin{prop}\label{4.7} Let $M$ be  a right $R$-module and $\mathcal{A}$, a class of artinian right $R$-modules and  closed under isomorphisms and summands. If every submodule of $M$ is $\mathcal{A}$-projective,  then the following conditions are equivalent:
\begin {enumerate}
\item $M$ is an $\mathcal{A}$-C3 module.
\item Every submodule $N \in \mathcal {A}$ of $M$  that is isomorphic to a
direct summand of $M$ is itself a direct summand.
\item Whenever  $X_1, X_2,\ldots, X_n$ are direct summands of $M$ and $X_1, X_2, \ldots, X_n \in \mathcal {A}$,  then $\sum_{i=1}^nX_i$ is a direct summand of $M$.
\end{enumerate}
\end{prop}
\begin{proof}
$(1) \Rightarrow (2)$. Let $M_1$ be submodule of $M$ and  isomorphic to a direct summand $M_2$ of $M$ and $M_1 \in \mathcal {A}$. Then $M = M_2 \oplus M_2'$. If $M_1 \subset M_2$, then by $M_2$ is artinian and $M_1 \cong M_2$, implies that $M_1 = M_2$. Let $M_1 \nsubseteq M_2$ and $\pi : M_2 \oplus M_2'\rightarrow M_2'$ be projection. According to the hypothesis, $\Ker (\pi_ {\mid M_1})$ is a direct summand of $M_1$. It follows that  $M_1 = M_1 \cap M_2 \oplus N_1$. Since $N_1 \cong \pi (M_1)$, $M_1 \cong M_2$, then there is an isomorphism $\phi: N '\rightarrow \pi(M_1)$, where $N '$ is a direct summand of $M_1$. Since $ \langle\phi\rangle \in \mathcal {A}$ and $\langle\phi\rangle \cap M_2 = 0$, $M_2 + \langle\phi\rangle = M_2 \oplus N_1$ is a direct summand of $M$. Therefore, $N_1$ is a non-zero direct summand of $M$. It is clear that $M_1 \cap M_2 \in \mathcal {A}$ and $M_1 \cap M_2$ is isomorphic to a direct summand of $M$. If $M_1 \cap M_2$ is not a direct summand of $M$, by using a argument that are similar to the argument presented above, we can show that $M_1 \cap M_2 = N_2 \oplus N_2'$, where $N_2\in \mathcal {A}$ is a non-zero direct summand of $M$ and  $N_2'\in \mathcal {A}$ is a submodule of $M$  isomorphic to a direct summand of $M$. Since each module of the class $\mathcal {A}$ is artinian, by conducting similar constructions continue for some $k$, we obtain a decomposition $M_1 = N_1 \oplus \ldots \oplus N_k$, where $N_i$ is a direct summand of $M$ and $N_i \in \mathcal{A}$ for each $i$. Since $M$ is an $\mathcal{A}$-C3 module, $N_1 \oplus N_2\oplus  \ldots \oplus N_k$ is a direct summand of $M$.

$(2)\Rightarrow(1)$. It is obvious.

$(1) \Rightarrow (3)$.
We prove this by induction on $n$.
When $ n = 2 $,  the assertion follows from Proposition \ref{4.5}. Suppose that the assertion is true for $n = k$. Let $X_1, X_2,\ldots, X_{k+1}$ be summands of $M$ and $X_1, X_2,\ldots, X_{k+1} \in \mathcal {A}$. Then there exists a  submodule $N$ of $M$ such that  $M=(\sum_{i=1}^kX_i)\oplus N$. Let $\pi: (\sum_{i=1}^kX_i)\oplus N\rightarrow N$ be the natural projection. As $\pi(X_{k+1})$ is  $\mathcal{A}$-projective, then $X_{k+1}=((\sum_{i=1}^kX_i)\cap X_{k+1}) \oplus S$  for some submodule $S$ of  $M$. Since the equivalence of (1) and (2), $\pi(X_{k+1})$ is a direct summand of $M$ and, therefore, $N=\pi(X_{k+1})\oplus T$ with  $T$ a submodule $M.$ It follows that $\sum_{i=1}^{k+1}X_i=(\sum_{i=1}^k X_i)\oplus \pi(X_{k+1})$ and $M=(\sum_{i=1}^kX_i)\oplus\pi(X_{k+1})\oplus T.$
Thus, $\sum_{i=1}^{k+1}X_i$ is a direct summand of $M.$
\end{proof}

\begin{rem}  Let $F$ be any nonzero free module over $\mathbb{Z}$ and $\mathcal{A}$,  a class of all free
$\mathbb{Z}$-modules. It is well known that $F$ is a quasi-continuous module and $F$ is not a continuous module. Thus, $F$ is an $\mathcal{A}$-C3 module and satisfies the property: there exists a submodule $N\in \mathcal {A}$ of $F$   that is isomorphic to a
direct summand of $F$ is not a direct summand.
\end{rem}

\begin{prop}\label{4.6} Let $M$ be  a right $R$-module and $\mathcal{A}$,  a class of right $R$-modules and  closed under isomorphisms and summands.
If every factor module  of $M$ is $\mathcal{A}$-projective,  then the following conditions are equivalent:
\begin{enumerate}
\item For any two direct summands $M_1, M_2$ of $M$ such that $M_1, M_2 \in \mathcal{A}$,   $M_1 + M_2$ is a direct summand  of $M$.
\item $M$ is an $\mathcal{A}$-C3 module.
\item For any decomposition $M=A_1\oplus A_2$ with $A_1\in \mathcal{A}$, then every homomorphism $f:A_1\to A_2$ has the image  a direct summand of $A_2$.
\item Every submodule $N \in \mathcal {A}$ of $M$  that is isomorphic to a
direct summand of $M$ is itself a direct summand.
\item Whenever $X_1, X_2, \ldots, X_n$ are direct summands of $M$ and $X_1, X_2, \ldots, X_n \in \mathcal {A}$,  then $\sum_{i=1}^nX_i$ is a direct summand of $M$.
\end{enumerate}

\end{prop}
\begin{proof} $(1) \Rightarrow (2)$  is obvious.

$(2) \Rightarrow (3) \Rightarrow (1)$ are  proved similarly to the argument proof of Proposition \ref{4.5}.

$(4) \Rightarrow (2)$ is obvious.

$(3) \Rightarrow (4)$. Let $\sigma: A\to B$ be an isomorphism with $A\in \mathcal{A}$ a summand of $M$ and $B\leq M$. We need to show that $B$ is a direct summand of $M$.
Write  $M=A\oplus T$ for some submodule $T$ of $M$. We have  $A/A\cap B$ is an image of $M$ and obtain that   $A\cap B$ is a direct summand of $A$.   Take  $A=(A\cap B)\oplus C$ for  some submodule $C$ of $A$. Now  $M=(A\cap B)\oplus (C\oplus T)$. Clearly,   $A\cap [(C\oplus T)\cap B]=0$  and  $B=(A\cap B)\oplus [(C\oplus T)\cap B]$. Let $H:=\sigma^{-1}((C\oplus T)\cap B)$. Then $H$ is a submodule of  $A$, $H\cap [(C\oplus T)\cap B]=0$ and  $A=H\oplus H'$ for some  submodule $H'$ of $H$. Note that  $M=H\oplus (H'\oplus T)$. Consider the projection  $\pi: M\to H'\oplus T$. Then $$H\oplus [(C\oplus T)\cap B]=H\oplus \pi((C\oplus T)\cap B).$$
By (3), the image of the homomorphism $\pi|_{(C\oplus T)\cap B}\circ\sigma|_{H}: H\to H'\oplus T$ is a direct summand of $H'\oplus T$ since $H$ is contained in $\mathcal{A}$. Write $H'\oplus T=\pi|_{(C\oplus T)\cap B}\sigma(H)\oplus K$ for some  submodule $K$ of $H'\oplus T$.  Then  $H'\oplus T=\pi((C\oplus T)\cap B)\oplus K$. It follows that $$M=H\oplus \pi((C\oplus T)\cap B)\oplus K=H\oplus [(C\oplus T)\cap B]\oplus K.$$
By the modular law,  $C\oplus T=[(C\oplus T)\cap B]\oplus [(H\oplus K)\cap (C\oplus T)]$. Thus $$\begin{array}{lll}
M&=(A\cap B)\oplus [(C\oplus T)\cap B]\oplus [(H\oplus K)\cap (C\oplus T)] \\
&=B\oplus [(H\oplus K)\cap (C\oplus T)].
\end{array}$$

The implication $(1) \Rightarrow (5)$ is proved similarly to the argument proof of Proposition \ref{4.7}.
\end{proof}

\begin{cor} The following conditions are equivalent for a module $M$:
\begin{enumerate}
\item For any semisimple submodules $A$, $B$ of $M$ with $A\cong B\subset_d M$, $A \subset_d M$.
\item For any semisimple summands $A, B$ of $M,$  $A+ B\subset_d M$.
\item For any semisimple summands $A, B$ of $M$ with $A\cap B=0,$ $A+ B\subset_d M$.
\item Whenever  $X_1,  \ldots, X_n$ are semisimple summands of $M$ and $X_1,  \ldots, X_n \in \mathcal {A}$,  then $\sum_{i=1}^nX_i$ is a direct summand of $M$.
\end{enumerate}
\end{cor}

\begin{cor} Let $Q$ be a quasi-injective module. If $X_1,\ldots, X_n$ are semisimple summands of $Q$, then $\sum_{i=1}^nX_i$ is a direct summand of $Q$.
\end{cor}

\begin{cor}[{\cite[Proposition 2.1]{CIYZ}}] The following conditions are equivalent for a module $M$:
\begin{enumerate}
\item For any simple submodules $A$, $B$ of $M$ with $A\cong B\subset_d M$, $A \subset_d M$.
\item For any simple summands $A, B$ of $M$ with $A\cap B=0,$  $A\oplus B\subset_d M$.
\item For any finitely generated semisimple submodules $A$, $B$ of $M$ with $A\cong B\subset_d M$, $A \subset_d M$.
\item For any finitely generated semisimple summands $A, B$ of $M$ with $A\cap B=0,$  $A\oplus B\subset_d M$.
\end{enumerate}
\end{cor}
\section{Characterizations of rings}

\begin{lem}\label{lem:criter} Let  $\mathcal{A}$ be a class of  right $R$-modules with local endomorphisms  and  closed under isomorphisms.   Assume that $K$ and $M$ are indecomposable right $R$-modules and  not  contained in  $\mathcal{A}$. Then
\begin{enumerate}
\item $N=M\oplus P$ is an  $\mathcal{A}$-D3 module for all
projective modules $P$.
\item $N=M\oplus E$ is an  $\mathcal{A}$-C3 module for all
injective modules $E$.
\item $N=M\oplus K$ is an $\mathcal{A}$-D3 module and
an $\mathcal{A}$-C3 module.
\end{enumerate}
\end{lem}

\begin{proof}
$(1)$ Let $N/A\cong S\subset_d N$ with $S\in  \mathcal{A}$. By \cite[%
Lemma 26.4]{AF}, there exist a direct summand $M_{1}$ of $M$ and a direct
summand $P_{1}$ of $P$ such that $N=S\oplus M_{1}\oplus P_{1}$. Write $%
P=P_{1}\oplus P_{2}$ for some submodule $P_{2}$ of $P$. Since $M$ is an
indecomposable module, we have either $M_{1}=0$ or $M=M_{1}$. If $M_{1}=0$,
then $N=S\oplus P_{1}=(M\oplus P_{2})\oplus P_{1}$ and it follows that $%
M\oplus P_{2}\cong S$, and hence $M\in  \mathcal{A}$ contradicting. So $M_{1}=M$. Then $N=S\oplus
(M\oplus P_{1})=(M\oplus P_{1})\oplus P_{2}$. This gives $S\cong P_{2},$ and
consequently $N/A\cong S$ is projective. Hence, $A$ is a direct summand of $%
N $ and $(1)$ holds.

$(2)$   Suppose  that $A$ is a submodule of $N$ such that $A\simeq S$ with $S$ a submodule of $N$ and  $S \in \mathcal{A}$ .
As in $(1),$ we see that $N=S\oplus M_{1}\oplus E_{1}$ with $M=M_{1}\oplus
M_{2}$ and $E=E_{1}\oplus E_{2}$. Also, as in (1), $M_{1}=M$. Therefore,
\begin{equation*}
N=S\oplus M\oplus E_{1}=M\oplus E=(M\oplus E_{1})\oplus E_{2}.
\end{equation*}
It follows that $S\simeq E_2$ is an injective module. Thus $A$ is a direct summand of $N$.

$(3)$ We show that $N$ has no a nonzero  direct summand $S$ with $S\in \mathcal{A}$. Assume on
the contrary that there exists a non-zero summand $S\subset_d  N$
with $S\in \mathcal{A}$. As, in $(1),$ $N=S\oplus M_{1}\oplus K_{1}$ with $%
M=M_{1}\oplus M_{2}$ and $K=K_{1}\oplus K_{2}$. Also, as in (1), $M_{1}=M$.
Therefore,
\begin{equation*}
N=S\oplus M\oplus K_{1}=M\oplus K.
\end{equation*}%
Since $K$ is indecomposable, $K=K_{1}$ or $K=K_{2}.$ If $K=K_{1},$ then $%
S\oplus M\oplus K=M\oplus K$ and consequently $S=0,$ a contradiction. If $%
K=K_{2},$ then $K_{1}=0$ and so $S\oplus M=M\oplus K.$ Therefore, $K\cong S$
and hence $K\in  \mathcal{A}$, a contradiction.
\end{proof}
Recall that a module is \textit{uniserial} if the lattice of its submodules
is totally ordered under inclusion. A ring $R$ is called right \textit{%
uniserial} if $R_R$ is a uniserial module. A ring $R$ is called \textit{%
serial} if both modules $_{R}R$ and $R_{R}$ are direct sums of uniserial
modules.
\begin{thm}\label{thm:4.2q}  Let $R$ be a right artinian ring and  $\mathcal{A}$,  a class of  right $R$-modules with local endomorphisms,  containing  all right simple right $R$-modules and  closed under isomorphisms.  If all right $R$-modules are $\mathcal{A}$-injective,  then the following conditions are equivalent for a ring $R$:
\begin{enumerate}
\item $R$ is a serial artinian ring with $J^{2}(R)=0$.
\item Every $\mathcal{A}$-C3 module is quasi-injective.
\item Every $\mathcal{A}$-C3 module is $C3$.
\end{enumerate}
\end{thm}

\begin{proof}
$(1) \Rightarrow (2)$ Assume that $R$ is an artinian serial ring with $%
J^{2}(R)=0 $. Then every right $R$-module is a direct sum of a semisimple
module and an injective module. Furthermore, every injective module is a
direct sum of cyclic uniserial modules. Let $M$ be an
$\mathcal{A}$-C3  module. We can write $M=(\oplus _{\mathcal{I}%
}S_{i})\oplus (\oplus _{\mathcal{J}}E_{j})$ where each $S_{i}$ is simple if $%
i\in \mathcal{I}$ and $\oplus _{\mathcal{J}}E_{j}$ is injective where each $%
E_{j}$ is cyclic uniserial non-simple if $j\in {J}$. Note that any $E_{j}$
has length at  2 by \cite[13.3]{DHSW}. We show that $M$ is a
quasi-injective module. To show that $M$ is quasi-injective, by \cite[%
Proposition 1.17]{MM} it suffices to show that $\oplus _{\mathcal{I}}S_{i}$ is
$\oplus _{\mathcal{J}}E_{j}$-injective. By \cite[Theorem 1.7]{MM}, $%
\oplus _{\mathcal{I}}S_{i}$ is $\oplus _{\mathcal{J}}E_{j}$-injective if
and only if $S_{i}$ is $\oplus _{\mathcal{J}}E_{j}$-injective for all $i\in
\mathcal{I}$. Furthermore, for any $i\in \mathcal{I}$, if $S_{i}$ is $E_{j}$%
-injective for all $j\in \mathcal{J}$, then $S_{i}$ is $\oplus _{\mathcal{J}%
}E_{j}$-injective by \cite[Proposition  1.5]{MM}. So, it suffices to show
that $S_{i}$ is $E_{j}$-injective for each $i\in \mathcal{I}$ and $j\in
\mathcal{J}$. Suppose that $E_j$ has a series $0\subset X\subset E_j$.
Let $f: A\to S_i$ be a homomorphism with $A\leq E_j$. If $A=0$ or $A=E_j$ then it is obvious that  $f$ is extended to a homomorphism from $E_j$ to $S_i$. Assume that $A=X$. If $f$ is non-zero, then $X\simeq S_i$.   As $M$ is an  $\mathcal{A}$-C3 module, $X$   is a direct summand of $%
M$. It follows that $X=E_{j}$, a contradiction. Hence $S_{i}$ is $E_{j}$%
-injective and so $M$ is quasi-injective.

$(2)\Rightarrow (3)$ This is clear.

$(3)\Rightarrow (1)$  Let $M$ be an indecomposable module. If $M\in \mathcal{A}$,
then it is quasi-injective. Now, suppose that $M\not\in \mathcal{A}$ and let $%
\iota :M\rightarrow E(M)$ be the inclusion. Then, by Lemma \ref{lem:criter}, $M\oplus E(M)$ is
$\mathcal{A}$-C3 and by assumption, $M\oplus E(M)$ is a $C3$-module. It
follows that $\im(\iota )$ is a direct summand of $E(M)$ by \cite[Proposition 2.3]%
{Amin}. Hence $M$ is injective. Inasmuch as every indecomposable right $R$%
-module is quasi-injective, we infer from \cite[Theorem 5.3]{Fuller2} that
$R$ is an artinian serial ring. By \cite[Theorem 25.4.2]{Faith}, every right
$R$-module is a direct sum of uniserial modules. Now, by \cite[13.3]{DHSW}, we only need to show that each uniserial module, say $M$, has length
at most 2. Suppose that $M$ has a series $0\subset X\subset Y\subset M$ of
length 3. Assume that $Y\in \mathcal{A}$. Then $X$ is $Y$-injective and hence $X$ is a direct summand of $Y$, a contradiction. It follows that $Y\not\in \mathcal{A}$. By Lemma \ref{lem:criter}, $M\oplus Y$ is an $\mathcal{A}$-C3 module and then, by hypothesis, is a C3-module. Consequently, the natural inclusion, $\eta :Y\longrightarrow M$
splits; i.e. $Y\subset_d M$ and so $Y=M,$ a contradiction. Hence, $%
R$ is an artinian ring with $J^{2}(R)=0.$
\end{proof}

\begin{thm}\label{thm:4.4q}  Let $R$ be a right artinian ring and  $\mathcal{A}$,  a class of  right $R$-modules with local endomorphisms,  containing  all right simple right $R$-modules and  closed under isomorphisms.    If all right $R$-modules are $\mathcal{A}$-projective, then the following conditions are equivalent for a ring $R$:

\begin{enumerate}
\item $R$ is a serial artinian ring with $J^{2}(R)=0$.
\item Every $\mathcal{A}$-D3  module is quasi-projective.
\item Every $\mathcal{A}$-D3  module is $D3$.
\end{enumerate}
\end{thm}

\begin{proof}By Lemma \ref{lem:criter} and  \cite[Theorem 4.4]{IKQY}.
\end{proof}

\begin{prop}  Let $\mathcal{A}$ be a class of right $R$-modules and  closed under isomorphisms and summands. Then the following conditions are equivalent:
\begin{enumerate}
\item  All modules  $A\in  \mathcal{A}$ are injective.
\item  Every right $R$-module   is  $\mathcal{A}$-C3.
\end{enumerate}
\end{prop}
\begin{proof} $(1)\Rightarrow (2)$ is obvious.
\vskip 0.1cm
$(2)\Rightarrow (1)$.  Suppose that  $A\in  \mathcal{A}$. Then by (2),  $A\oplus E(A)$ is an $\mathcal{A}$-C3 module. Call $\iota: A\to E(A)$ the inclusion map.  By Proposition \ref{pro:ad2}, $\im(\iota)=A$ is a direct summand of $E(A)$. Thus $A=E(A)$ is an injective module.
\end{proof}
\begin{cor}[\cite{CIYZ}] The following conditions are equivalent for a ring $R$:
\begin{enumerate}
\item $R$ is a right V-ring.
\item Every right $R$-module is simple-direct-injective.
\end{enumerate}
\end{cor}

\begin{prop}  Let $\mathcal{A}$ be a class of right $R$-modules and  closed under isomorphisms and summands. Then the following conditions are equivalent:
\begin{enumerate}
\item  All modules  $A\in  \mathcal{A}$ are projective.
\item  Every right $R$-module   is  $\mathcal{A}$-D3.
\end{enumerate}
\end{prop}
\begin{proof} $(1)\Rightarrow (2)$. Assume that $M$ is a right $R$-module.  Let    $M_1, M_2$ be submodules of $M$ with  $M/M_1, M/M_2\in \mathcal{A}$ and $M=M_1+M_2$. It follows that $M/M_1, M/M_2N$ are projective modules and the following isomorphism   $$M/(M_1\cap M_2)=(M_1+M_2)/(M_1\cap M_2)\simeq M/M_1\times  M/M_2.$$

 Then $M/(M_1\cap M_2)$ is a projective module.  We deduce that  $M_1\cap M_2$ is a direct summand of $M$. It shown that  $M$ is  an $\mathcal{A}$-D3 module.

$(2)\Rightarrow (1)$.  Suppose that  $A\in  \mathcal{A}$. Call $\varphi: R^{(I)}\to A$ an epimorphism.  Then  $R^{(I)}\oplus A $ is an $\mathcal{A}$-D3 module. By Proposition \ref{pro:add1},  $A$ is isomorphic to a direct summand of $R^{(I)}$. Thus $A$ is a projective module.
\end{proof}
\begin{cor}[\cite{IKQY}] The following conditions are equivalent for a ring $R$:
\begin{enumerate}
\item $R$ is a semisimple artinian ring.
\item Every right $R$-module is simple-direct-projective.
\end{enumerate}
\end{cor}

Let $M$ be a right $R$-module.  $M$ is called \emph{regular} if every cyclic submodule of $M$ is a direct summand. A right $R$-module is called \emph{$M$-cyclic} if it is isomorphic to a factor module of $M$.

\begin{lem}\label{lem:vmodule}Let $F$ be a regular module. Assume that  $A\not=0$ is a small finitely generated submodule of the factor module $F/F_0$ for some submodule $F_0$ of $F$ and $\mathcal {A}$ the class of all modules isomorphism   to $A$. Then there exists a $F$-cyclic module $M$ and satisfies the property: there is  a  submodule $N \in \mathcal {A}$ of $M$  that is isomorphic to a direct summand of $M$ and not  a direct summand.
\end{lem}

\begin{proof} By the hypothesis we have  $((x_1R+x_2R+ \cdots+ x_mR) + F_0)/F_0=A$ for some $x_1,x_2,\dots,x_m$ of $F$.  Since $F$ is a regular module,   $x_1R+x_2R+ \cdots+ x_mR= \pi (F)$, where $\pi \in \End(F)$ and $\pi^2 = \pi$. Since $A$ is a small submodule of $F/F_0$, we have $F / F_0 = ((1- \pi) F + F_0)/F_0$. It follows that  there exist epimorphisms $f_1: \pi (F) \rightarrow A,$ $f_2: (1- \pi) (F) \rightarrow F / F_0$. It is easy to check   $A \oplus (F / F_0)$ is an $F$-cyclic module. Call $M=A \oplus (F / F_0)$.  Thus, the module  $N:=0 \oplus A\simeq A$  is  not a direct summand of $M$ and   isomorphic to a direct summand of $M$.
\end{proof}
A module $M$ is called a {\it V-module}   if every simple module in $\sigma[M]$
 is $M$-injective (see \cite{W}).  $R$ is called a  right {\it V-ring}  if the right module $R_R$ is a V-module.
\begin{thm}\label{thm:4.10a}  The following conditions are equivalent for a regular module $F$:
\begin{enumerate}
\item $F$ is a $V$-module.
\item Every $F$-cyclic module $M$  is an $\mathcal{A}$-C3 module where $\mathcal{A}$ is the class of all simple submodules of $M$.
\end{enumerate}
\end{thm}
\begin{proof} The implication $(1) \Rightarrow (2)$ is obvious.

$(2) \Rightarrow (1)$. Let $S \in \sigma [F]$ is a simple module and $E_F (S)$ is the injective hull of   $S$ in the category $\sigma [F]$. Assume that $E_F (S) \neq S$. As  $E_F(S)$ is generated by   $F$, there exists a homomorphism $f: F \rightarrow E_F (S)$ such that $f (F) \neq S$. Then $S$ is a small submodule of $f (F)\simeq F/\Ker(f)$. Call $\mathcal {A}$ the class of all modules isomorphism  to $S$. By Lemma \ref{lem:vmodule},  there exists a  $F$-cyclic module $M$ and  satisfies the property: there is  a  submodule $N \in \mathcal {A}$ of $M$  that is isomorphic to a direct summand of $M$ and  not  a direct summand. We infer from Proposition \ref{4.6}  that  $M$ is not  an $\mathcal{A}$-C3 module. This contradicts the condition of (2).
\end{proof}

\begin{cor}[{\cite[Theorem 4.4.]{CIYZ}}] A regular ring R is a right V-ring if and only if every cyclic right R-module is simple-direct-injective.
\end{cor}

\bibliographystyle{amsplain}

\end{document}